\documentclass[11pt]{article}
\usepackage{amsmath,amssymb,amsfonts,amsthm}
\usepackage{bm} 
\usepackage{physics} 
\usepackage{mathtools} 
\usepackage{ulem}

\usepackage{enumerate} 
\usepackage{geometry} 
\usepackage{microtype}
\usepackage{hyperref} 
\hypersetup{colorlinks=true, linkcolor=blue, citecolor=blue, urlcolor=blue}

\newtheorem{theorem}{Theorem}[section] 
\newtheorem{lemma}[theorem]{Lemma} 

\newtheorem{assumption}[theorem]{Assumption}
\newtheorem{example}[theorem]{Example} 

\newcommand{\R}{\mathbb{R}}
\newcommand{\cL}{\mathcal{L}}
\newcommand{\ao}[1]{{#1}}
\newcommand{\cB}{{\cal B}}

\theoremstyle{definition} 

\DeclareMathOperator{\E}{\mathbb{E}} 

\newcommand{\T}{\mathsf{T}} 

\newcommand{\cS}{\mathcal{S}} 

\begin{document}

\title{Interpretable Gradient Descent for Kalman Gain}
\author{M.A.~Belabbas\thanks{University of Illinois at Urbana-Champaign, {\tt belabbas@illinois.edu}.} $~$and A.~Olshevsky\thanks{Boston University, {\tt alexols@bu.edu}.\\ Authors names are ordered alphabetically.}}
\date{\today}
\maketitle

\begin{abstract}
We derive a decomposition for the gradient of the innovation loss with respect to the filter gain in a linear time-invariant system, decomposing as a product of an observability Gramian and a term quantifying the ``non-orthogonality" between the  estimation error and the innovation.  We leverage this decomposition to give a convergence proof of gradient descent to the optimal Kalman gain,   specifically identifying how recovery of the Kalman gain depends on a non-standard observability condition, and obtaining an interpretable geometric convergence rate. 
\end{abstract}

\section{Introduction}

While the optimal Kalman gain for a linear time–invariant system can be obtained in closed form from the algebraic Riccati equation, gradient-based strategies to compute the Kalman gain remain interesting for two main reasons.  First, in time-varying or adaptive filtering environments, incremental gradient updates permit the gain to be revised continuously, so the filter can respond to evolving system dynamics without repeatedly solving a full Riccati equation.  Second, when the system model is only partially known—or entirely unknown—but measurement data continue to arrive, stochastic gradient methods may be used to directly to identify the gain in a data-driven manner.

 We adopt the expected mean–square magnitude of the innovation —i.e., the one-step-ahead prediction error—as our cost for gradient descent, a quantity we refer to as the \emph{innovations loss}. This is the most direct way to measure filter performance and, since it relies solely on observable data, so it remains well defined even when the true state is unknown. 

 The problem of analyzing gradient descent on the innovations loss has, to our knowledge, only been previously considered in \cite{talebi2023data}. A Kalman filter in one-step lookeahead form was considered in \cite{talebi2023data}, and it was shown that, under observability of the underlying system, gradient descent and its stochastic variants converge to the  Kalman gain. A closely related paper \cite{li2023policy} considered a slightly different cost function, using look-ahead in the innovations, and also obtained convergence to the Kalman gain.

 There is also considerable parallel literature on the related problem of computing the optimal gain for LQR using gradient descent (e.g., \cite{fazel2018global, fatkhullin2021optimizing, mohammadi2020linear,talebi2024policy}). However, as explained in \cite{talebi2023data}, although some form of duality between controllability and observability holds here, one generally cannot derive results for computing the Kalman gain from corresponding results for LQR due because key nonsingularity assumptions are not preserved in the  transformation. 

In this paper we consider gradient descent to recover the Kalman filter in the more standard formulation. Unlike in \cite{talebi2023data}, we show that if we assume only observability of the underlying system, spurious stationary points do exist. Our starting point for analysis is a new observation that the gradient of the innovation loss is a product of an observability Gramian and a term measuring the violation of the orthogonality principle for Kalman filters. Building on this decomposition, we give a slightly non-standard observability condition that rules out stationary points and ensures convergence of gradient descent at a geometric rate to the Kalman gain. Finally, we show that this rate is interpretable, being a product of two terms: a measure of  how close the trajectory can come to losing observability  and a measure of how steeply violation of the orthogonality principle forces the innovations loss to rise.





\section{Problem Statement\label{sec:setting}}

We consider a discrete-time linear time-invariant system described by the state-space equations
\begin{align}
x_{t+1} &= A\,x_t + w_t, \label{eq:state_dyn}\\
y_{t+1} &= C\,x_{t+1} + v_{t+1}, \label{eq:meas_dyn}
\end{align}
where $x_t \in \mathbb{R}^n$ is the state vector, $y_t \in \mathbb{R}^p$ is the measurement vector. We assume the process noise $w_t$ and measurement noise $v_t$ are  independent, identically distributed zero-mean random vectors with covariances $\E[w_t w_t^\T] = Q_w \succ 0$ and $\E[v_t v_t^\T] = R_v \succ 0$, respectively.

We employ a linear filter to estimate the state from observations:
\begin{align}
\hat x_{t+1}^- &= A\,\hat x_t, \label{eq:pred_state}\\
\delta_{t+1} &= y_{t+1} - C\,\hat x_{t+1}^-, \label{eq:innov}\\
\hat x_{t+1} &= \hat x_{t+1}^- + L\,\delta_{t+1}. \label{eq:update_state}
\end{align}
Here $\hat x_{t+1}$ is the estimate of $x_{t+1}$, $\delta_{t+1}$ is the innovation  and $L \in \mathbb{R}^{n \times p}$ is the filter gain matrix. If $L$ is chosen so as to minimize the expected mean square state estimation error in steady-state, the resulting filter is the Kalman filter, and we denote the corresponding optimal gain as $L_{\rm KF}$. 

However, the performance measure we consider here is  $E \|\delta_{t}\|_2^2$ in steady-state as a function of a general gain $L$. Specifically, let 

$$F(L):=(I-LC)A$$ and define the set
$$ \cS:=\{ L \mid  \rho(F(L))<1 \},$$ where $\rho(\cdot)$ is the spectral radius, be the set of stabilizing gains for the pair $(A,C)$. We define 
the steady-state innovation covariance as
\[ \Sigma_\delta(L) = \lim_{t\to\infty} \E[\delta_{t+1} \delta_{t+1}^\T]. \] 
It is known that this limit exists if $L \in \cS$.  Our error metric will be the function

$$ J_{\rm innov}: \cS \to \R  = \Tr( \Sigma_{\delta}(L)). $$

Our goal in this work is to answer the following three related questions:
\begin{enumerate} 
\item Under what condition does the gradient descent dynamics \[ \frac{d}{dt} L = - \nabla_{L} J_{\rm innov}(L),\] converge to the Kalman gain, i.e.,  when does it hold that 
\[ \lim_{t \to \infty} L(t) = L_{KF}\]

\item Is there an explicit expression for $\nabla_{L} J_{\rm innov}(L)$ which provides an intuitive explanation for the absence of spurious stationary points and the convergence of gradient descent above?

\item Can we answer the previous two questions in terms of easily interpretable properties of the dynamics of~\eqref{eq:state_dyn} and~\eqref{eq:meas_dyn}?
\end{enumerate} 


\section{The  Gradient of the Innovations Loss} \label{sec:gradient_derivation}

We now introduce some light notation that will allow us to state our first result. As is standard, the state estimation error is denoted by
 \[ e_t := x_t - \hat x_t, \] 
  and the steady-state cross-covariance between the  estimation error $e_{t}$ and the innovation $\delta_{t}$ by 
\[ K(L) := \lim_{t\to\infty} \E[e_{t} \delta_{t}^\T], \]
where the dependence on $L$ is through the definitions of the estimation error and innovation. 
It is well-known that 
\[ K(L_{KF}) = 0,\] where, recall, $L_{KF}$ is the Kalman gain. This is sometimes called {\em the orthogonality principle} for Kalman filtering. Intuitively, we can think of the size of $K(L)$ as a measure of violation of this principle.

Additionally, we denote by $W_o(L)$   observability Gramian of the system $(F(L), CA)$. Formally,  provided $\rho(F(L)) < 1$, the quantity $W_o(L)$ is the unique positive semidefinite solution to the Lyapunov equation:
\begin{equation} \label{eq:Wo_lyap_def}
W_o(L) = F(L)^\T W_o(L) F(L) + A^\T C^\T C A.
\end{equation} 
With these definitions in place, we can state our first main result.

\begin{theorem} \label{thm:grad1_main}
The gradient of $J_{\rm innov}(L)$ with respect to the standard Euclidean inner product is given by
\[ \nabla_{L} J_{\rm innov}(L) = - 2 W_o(L) K(L). \]
\end{theorem}

This expression provides an interpretation of the meaning of this gradient: it is a product of two terms, the first of which measures the observability of the underlying system and the second of which measures a violation of the orthogonality principle. This perspective will prove particularly valuable in the subsequent section, where we analyze gradient descent to compute the Kalman gain. What is specifically novel about this expression relative to the previous literature (e.g.,~\cite[Section D.3]{talebi2023data}) is the clean form and interpretability of the result.




\subsection{Proof of Theorem \ref{thm:grad1_main}}

Our first step is to write down an expression for $J_{\rm innov}(L)$ and a number of related quantities which will be useful later. All of these are standard in the Kalman filtering literature, but we state them here to be self-contained and standardize notation. 

We begin by writing down the dynamics of the error. A straightforward calculation shows that: 
\[ e_{t+1} = (I - LC)A e_t + (I - LC)w_t - L v_{t+1}. \]
We write this as
\[ e_{t+1} = F(L) e_{t} + \eta_t(L), \]
where we introduced the {\em effective noise} $\eta_t(L)$, defined as 
\begin{equation} \label{eq:eff_noise_def}
\eta_t(L) := (I - LC)w_t - L v_{t+1}.
\end{equation}
The covariance  $Q_\eta(L) := \E[\eta_t(L) \eta_t(L)^\T]$ of the effective noise is given by
\begin{equation} \label{eq:Q_eta_def}
Q_\eta(L) = (I-LC)\,Q_w\,(I-LC)^\T + L\,R_v\,L^\T.
\end{equation}
We show later in Lemma \ref{rem:Qeta_posdef} that $Q_{\eta}(L)$ is positive definite.

If the filter is stable, i.e., if $L \in \cS$,  the steady-state error covariance $$P(L) := \lim_{t\to\infty}\E[e_t e_t^\T],$$ exists and is the unique, symmetric positive semidefinite solution to the discrete algebraic Lyapunov equation
\begin{equation}\label{eq:P_lyap}
P(L) = F(L)\,P(L)\,F(L)^\T + Q_\eta(L).
\end{equation}
If, moreover, $Q_\eta(L) \succ 0$ then  $P(L)$ is also positive definite.

\bigskip 

A related quantity is the {\em a priori estimation error} defined as  $$e_{t+1}^- = x_{t+1} - \hat{x}_{t+1}^-.$$ It is related to the estimation error by $$e_{t+1}^- = A e_t + w_t,$$ and to the innovation by \begin{equation} \label{eq:delta} \delta_{t+1} = C e_{t+1}^- + v_{t+1}.
\end{equation} 
From the above relation for $e_{t+1}^-$, we easily obtain that the steady-state covariance of the a priori estimation error is given by
\begin{equation} \label{eq:pminus} P^-(L):= \lim_{t\to\infty}\E[e_t^- e_t^{-\T}] = A P(L) A^\T + Q_w. \end{equation} 
From~\eqref{eq:delta} and~\eqref{eq:pminus}, the steady-state innovation covariance is then
\[ \Sigma_\delta(L) = \lim_{t\to\infty} \E[\delta_{t} \delta_{t}^\T] = C P^-(L) C^\T + R_v = C (A P(L) A^\T + Q_w) C^\T + R_v. \]

We define the cost function $J_{\rm innov}(L)$ as the trace of this matrix:
\begin{equation} \label{eq:jinnov_def}
J_{\rm innov}(L) = \Tr(C A P(L) A^\T C^\T) + \Tr(C Q_w C^\T) + \Tr(R_v).
\end{equation}

We next require several intermediate lemmas before proceeding. Our first lemma establishes the positive definiteness of the effective noise covariance.

\begin{lemma}\label{rem:Qeta_posdef}
For any gain matrix $L$, we have that  $Q_\eta(L)$ is positive definite. 
\end{lemma}
\begin{proof} 
Observe that
\[ z^\T Q_\eta(L) z = z^\T (I-LC)Q_w(I-LC)^\T z + z^\T L R_v L^\T z = \|(I-LC)^\T z\|_{Q_w}^2 + \|L^\T z\|_{R_v}^2, \]
where $\|x\|_M^2 = x^\T M x$. This vanishes when $L^\T z=0$ and $z-C^\T L^\T  z=0$. This is possible only if $z=0$, which proves the claim.
\end{proof}
Next, we derive an expression for the cross-covariance $K(L)$.

\begin{lemma} \label{lem:covlemma_main}
For $L \in \cS$, it holds that 
\[ K(L) = F(L) P(L) A^\T C^\T + (I-LC) Q_{w} C^\T - L R_{v}. \]
\end{lemma}
\begin{proof} The update equation $e_{t+1} = e_{t+1}^- - L \delta_{t+1}$ and~\eqref{eq:delta} together yield 
\[ e_{t+1} = e_{t+1}^- - L(C e_{t+1}^- + v_{t+1}) = (I-LC) e_{t+1}^- - L v_{t+1}. \]
Now, we compute the cross-covariance $K(L) = \E[ e_{t+1} \delta_{t+1}^\T ]$:
\begin{align*}
K(L) &= \E[ ((I-LC) e_{t+1}^- - L v_{t+1}) (C e_{t+1}^- + v_{t+1})^\T ] \\
&= (I-LC) \E[e_{t+1}^- (e_{t+1}^-)^\T] C^\T - L \E[v_{t+1} v_{t+1}^\T] \\ 
&= (I-LC) P^-(L) C^\T - L R_v \\
&= (I-LC) A P(L) A^\T C^\T + (I-LC) Q_w C^\T - L R_v,
\end{align*} where the last step used~\eqref{eq:pminus}. Recalling the definition that $F(L)=(I-LC)A$ concludes the proof. 
\end{proof}

The following lemma provides the core of the gradient calculation. We omit writing the dependence of $F(L), P(L), Q_\eta(L)$ and$ W_o(L)$ on $L$ to save space.
Recall that the gradient of a function $J:\R^{p\times q} \to \R$ with respect to the Frobenius inner product is defined as the vector field $\nabla J$ satisfying
$$
\Tr(\nabla J(L)^\T V) = dJ(L) \cdot V \mbox{ for all } V \in \R^{n \times p},
$$
where $dJ(L) \cdot V$ is the differential of $J$ at $L$ evaluated in the direction $V \in \R^{p \times q}$.
\begin{lemma} \label{lem:gradlemma_main} 
For $L \in \cS$, we have
\[ \nabla_L J_{\rm innov}(L) \;=\; -2 W_o ( F P A^\T C^\T + (I-LC) Q_w C^\T - L R_v ). \]
\end{lemma}

\begin{proof}
Throughout this proof we work in the framework of \emph{matrix differentials}.  
For a smooth matrix-valued mapping $F:\mathbb{R}^{m\times n}\to\mathbb{R}^{p\times q}$ we write with some abuse of notation 
$dF(X)$ for the differential of $F$ evaluated at $X$ and applied to an arbitrary vector $dX$; 
namely, $dF(X)$ is so that   $F(X+dX)=F(X)+dF(X)+ o(\|dX\|)$.

From~\eqref{eq:jinnov_def}, $J_{\rm innov}(L) = \Tr(C A P A^\T C^\T) + \text{const}$. Thus, the differential $dJ_{\rm innov}$ is
\begin{equation} \label{eq:dJ_dp_main}
dJ_{\rm innov}  = \Tr(C A (dP) A^\T C^\T) = \Tr(A^\T C^\T C A \, dP) = \Tr(M dP),
\end{equation}
where $M = A^\T C^\T C A$. Now  the observability  Gramian $W_o$ obeys the Lyapunov equation
\begin{equation}\label{eq:adjoint_lyap_main}
W_o - F^\T W_o F = M.
\end{equation} With this identity in place, let us compute an expression for $dP$ which we will then plug into~\eqref{eq:dJ_dp_main}. Differentiating~\eqref{eq:P_lyap} we obtain
\begin{equation}\label{eq:diff_lyap_main}
dP - F (dP) F^\T = (dF) P F^\T + F P (dF)^\T + dQ_\eta.
\end{equation}
We multiply \eqref{eq:diff_lyap_main} by $W_o$ and take the trace:
\[
\Tr(W_o (dP - F (dP) F^\T)) = \Tr(W_o ((dF) P F^\T + F P (dF)^\T + dQ_\eta)).
\]
The left-hand side (LHS) simplifies using the cyclic property of trace and~\eqref{eq:adjoint_lyap_main} to
\[
\text{LHS} = \Tr(W_o dP) - \Tr(W_o F (dP) F^\T) = \Tr(W_o dP) - \Tr(F^\T W_o F dP)   = \Tr(M dP).
\] and thus $\Tr(M dP) = \Tr(W_o ((dF) P F^\T + F P (dF)^\T + dQ_\eta))$. Plugging this into~\eqref{eq:dJ_dp_main}, we  obtain
\begin{equation}\label{eq:dJ_final_form_main}
dJ_{\rm innov} = \Tr\bigl(W_o ((dF) P F^\T + F P (dF)^\T + dQ_\eta)\bigr).
\end{equation}

We next evaluate $dQ_\eta$ and $dF$.
For the latter, using $F=(I-LC)A$, we obtain  $dF = -(dL) C A$.
From $Q_\eta = (I-LC)\,Q_w\,(I-LC)^\T + L\,R_v\,L^\T$, we have
\[
dQ_\eta = -(dL) C Q_w (I-LC)^\T - (I-LC) Q_w C^\T (dL)^\T + (dL) R_v L^\T + L R_v (dL)^\T.
\]
Plugging these into~\eqref{eq:dJ_final_form_main}, we get

\begin{multline*}
dJ_{\rm innov} = \Tr\Biggl(W_o \Biggl( (-(dL) CA) P F^\T + F P (-(dL) C A)^\T 
-(dL) C Q_w (I-LC)^\T \\ - (I-LC) Q_w C^\T (dL)^\T 
+ (dL) R_v L^\T + L R_v (dL)^\T \Biggr) \Biggr).
\end{multline*}

We now use standard trace identities
as well as the symmetry of $W_o,P, R_v$ to move $dL$ to the right   and obtain after a short calculation
\[ dJ_{\rm innov} = \Tr \Bigl( \bigl( -2 C A P F^\T W_o -2 C Q_w (I-LC)^\T W_o + 2 R_v L^\T W_o \bigr) dL \Bigr). \]

Since $dL\cdot V = V$, we have shown that
$$d J_{\rm innov} \cdot V  = \Tr  \Bigl( \bigl( -2 C A P F^\T W_o -2 C Q_w (I-LC)^\T W_o + 2 R_v L^\T W_o \bigr) V \Bigr),$$ from which we conclude that
\begin{align*}
\nabla_L J_{\rm innov} &= \left[ -2 C A P F^\T W_o -2 C Q_w (I-LC)^\T W_o + 2 R_v L^\T W_o \right]^\T \\
&= -2 W_o ( F P A^\T C^\T + (I-LC) Q_w C^\T - L R_v ).
\end{align*}
This concludes the proof.
\end{proof}

\begin{proof}[Proof of Theorem \ref{thm:grad1_main}]
The result follows directly by comparing the expression for the gradient from Lemma \ref{lem:gradlemma_main}
with the expression for the cross-covariance $K(L)$ from Lemma \ref{lem:covlemma_main}.
\end{proof}

\section{Convergence of Gradient Descent}

We next focus our attention to the convergence of the gradient descent flow of $J_{\rm innov}$, which is the flow associated with the differential equation
\begin{equation} \label{eq:grad_flow_main}
 \frac{d}{dt} L(t) = - \nabla  J_{\rm innov}(L(t)) = 2 W_o(L(t)) K(L(t)).
\end{equation}

We make the following assumption.

\begin{assumption} \label{assum:obs_main}
The pair $(A, CA)$ is observable.
\end{assumption}

This assumption might look unnatural at first glance, given that the Kalman filter theory is usually developed using only the much weaker assumption that $(A,C)$ is detectable. However, it should make sense given the  second equality of~\eqref{eq:grad_flow_main}: if we want to rule out all stationary points except $L_{KF}$, we do want to make sure $W_{o}$ is not singular, at least at the initial point. 

We also need the following standard assumption to ensure existence and uniqueness of the Kalman filter. 

\begin{assumption} \label{assum:stabilizable_main}
The pair $(A, Q_w^{1/2})$ is stabilizable.
\end{assumption}

With these assumptions in place, we will prove the following theorem, which is our second main result. We show that the gradient flow of $J_{\rm innov}$ converges to the Kalman gain when initialized at an arbitrary $L(0)\in \cS$; said otherwise, gradient descent recovers the Kalman gain. 

To state the result precisely, we introduce the following level set: given $L(0) \in \cS$, we define
$$\mathcal{L}_0 =  \{ L   \in \cS \mid  J(L) \leq J(L(0)) \}.$$

We furthermore introduce the constants
\[ \kappa = \inf_{ L \in \mathcal{S}_0 } \lambda_{\min} \left( W_o(L) \right),\] and
\[ c = \sup_{ \{ L \in \mathcal{L}_0 \setminus \{ L_{KF} \} \}} \frac{J_{\rm innov}(L) - J_{\rm innov}(L_{KF})}{\|K(L)\|_F^2}.\]

We can now state the main result of this section.
\begin{theorem}[Convergence of Gradient Descent] Assume that  Assumptions  \ref{assum:obs_main} and \ref{assum:stabilizable_main}\label{thm:convergence_main} hold. Let $L(t)$ be the solution to~\eqref{eq:grad_flow_main}.  Then  $$\lim_{t \to \infty} L(t) = L_{KF}.$$
Moreover, it holds that 
\begin{enumerate} 
\item $\kappa > 0$ and $  c < \infty$
\item  The cost satisfies
\[ J_{\rm innov}(L(t)) - J_{\rm innov}(L_{KF}) \leq e^{-4 (\kappa^2/c)t} \left( J_{\rm innov}(L(0)) - J_{\rm innov}(L_{KF})\right).\]
\end{enumerate}
\end{theorem}


One novelty of this theorem 
is the identification of 
Assumption \ref{assum:obs_main}  that $(A,CA)$ is observable as a requirement for convergence of gradient descent to recover Kalman gain in the standard filtering setting of Section \ref{sec:setting}.
Another is 
the easily interpretable rate of $\kappa^2/c$, which is the product of the potential observability loss over the entire trajectory $\kappa$ and how steeply violation of the orthogonality principle forces the loss to rise $c$.  

We may contrast this theorem with the main result of \cite{talebi2023data}, which considered the Kalman filter in the {\em one-step lookahead} form. In that setting, the observability of $(A,C)$ was sufficient for recovery. The following example shows this condition is not enough for Kalman gain recovery in our setting.  

\begin{example}[Spurious stationary points when $(A,C)$ is observable but $(A,CA)$ is not] \label{ex:spurious}
Consider the system with matrices:
\[ A = \begin{pmatrix} 0 & 1 \\ 0 & 0 \end{pmatrix}, C = \begin{pmatrix} 1 & 0 \end{pmatrix}, Q_w = I_2, R_v = 1. \]
The pair $(A,C)$ is observable. However, $CA = \begin{pmatrix} 0 & 1 \end{pmatrix}$, and the observability matrix for $(A, CA)$ is $\begin{psmallmatrix} CA \\ CA^2 \end{psmallmatrix} = \begin{psmallmatrix} 0 & 1 \\ 0 & 0 \end{psmallmatrix}$, which is of rank $1$. So, $(A,CA)$ is not observable.
Set $L = \begin{psmallmatrix} l_1 \\ l_2 \end{psmallmatrix}$. The error transition matrix is $F(L) = (I-LC)A = \begin{pmatrix} 0 & 1-l_1 \\ 0 & -l_2 \end{pmatrix}$. Stability of $F(L)$ requires $|l_2| < 1$.
Since 
\[
\delta_{t+1}
= CAe_t + Cw_t + v_{t+1}
= (e_{t,2} + w_{t,1}) {\bf e}_1 + v_{t+1},
\] where ${\bf e}_i$ is the canonical basis vector with entry $i$ equal to one and all other entries zero, 
and \(w_{t,1}\sim\mathcal{N}(0,1)\), \(v_{t+1}\sim\mathcal{N}(0,1)\) are uncorrelated with \(e_t\), it follows that
\[
J_{\rm innov}(L)
= \Tr\bigl(\Sigma_{\delta}(L)\bigr)
= \E\bigl[e_{t,2}^2\bigr] + 2.
\]
Let 
\[
c = \lim_{t\to\infty}\E\bigl[e_{t,2}^2\bigr].
\]
The  \((2,2)\)-entry of~\eqref{eq:P_lyap} (with \(F(L) = \begin{pmatrix}0 & 1-l_1\\0 & -\,l_2\end{pmatrix}\), \(Q_w = I\), \(R_v=1\)) yields
\[
c 
= (l_2)^2\,c \;+\; \bigl(2\,l_2^2 + 1\bigr)
\quad\Longrightarrow\quad
c = \frac{\,1 + 2\,l_2^2\,}{\,1 - l_2^2\,},
\quad |l_2|<1.
\]
Therefore,
\[
J_{\rm innov}(L)
= c + 2
= \frac{\,1 + 2\,l_2^2\,}{\,1 - l_2^2\,} \;+\; 2,
\quad \text{for } |l_2|<1.
\]
This cost depends only on $l_2$ and  $\nabla J_{\rm innov}(L) = \begin{psmallmatrix} 0 \\ \frac{6l_2}{(1-l_2^2)^2} \end{psmallmatrix}$. Critical points occur when $l_2=0$, i.e., all points $L = (l_1, 0)^\T$, for any $l_1 \in \mathbb{R}$, are critical points. In particular, gradient descent initialized at any of these stationary points stays there, so that clearly $L_{KF}$ is not always recovered. It is also worth noting that the level sets of the cost $J_{\rm innov}(L)$ are {\em not} compact in this example. 
\end{example}

We now proceed to prove Theorem \ref{thm:convergence_main}. A key element of the proof is the coercivity of $J_{\rm innov}(L)$ under the assumption that $(A,CA)$ is observable (and note that the example just given also shows that $J_{\rm innov}(L)$ is not coercive under the weaker assumption that $(A,C)$ is observable). 
Recall that $\mathcal{S}$ is the set of stabilizing gains: $\mathcal{S} := \{ L \in \mathbb{R}^{n \times p} \mid \rho(F(L)) < 1 \}.$

\begin{lemma}[The Coercivity Lemma] \label{Lemma:coercivity} Suppose Assumption   \ref{assum:obs_main} holds.
For any sequence $L_k \in \mathcal{S}$ such that either $\|L_k\|_F \to \infty$ or $L_k$ approaches the boundary of $\mathcal{S}$ (i.e., $\rho(F(L_k)) \to 1^-$), it holds that $\lim_{k \to \infty} J_{\rm innov}(L_k) = \infty$.
\end{lemma} 

Due to the importance of this lemma and its tricky technical analysis, we relegate its proof to the next section. 

We will also  need the following technical lemma. 

\begin{lemma} \label{lemma:ratio_quadratic} Let $M,N$ be symmetric matrices such that for all $x \in \R^n$,  $x^T N x = 0$ implies $x^T M x=0$. Then there exists $\beta < \infty$ such that 
\[ x^T M x \leq \beta x^T N x\]
\end{lemma} 

\begin{proof} Let $N^{\dagger}$ be the Moore-Penrose pseudoinverse of $N$ and set $C = N^{\dagger/2} M N^{\dagger/2}$. We argue that $\beta = \lambda_{\rm max}(C) $ works. Indeed, for any $x$, decompose it as $x=y+z$ with $y \in \ker(N)$ and $z \in \ker(N)^\perp = \operatorname{range}(N)$. We have:
\[ x^T M x = z^T M z = (N^{1/2} z)^T C N^{1/2} z \leq \lambda_{\rm max}(C) z^T N z = \beta x^T N x,\] where we used that $N^{1/2} N^{\dagger/2} z= N^{\dagger/2} N^{1/2} z = z$ for any $z \in {\rm range}(N)$.  This
proves the result.
\end{proof} 

With all this in place, we are now ready to prove the main result of this section.

\begin{proof}[Proof of Theorem \ref{thm:convergence_main}]  
By Lemma \ref{Lemma:coercivity}, the level set $\cL_0 = \{ L \mid J(L) \leq J(L(0)) \}$ is compact. Since $(d/dt) J(L(t)) \leq 0$, we conclude that the trajectory $L(t)$ remains in this compact set. 

We use LaSalle's principle with respect to the Lyapunov function $V(L) = J_{\rm innov}(L)$. As noted, the trajectory remains within $\cL_0$, which is a compact subset of $\mathcal{S}$.  Furthermore, $V(L)$ is continuously differentiable on $\mathcal{S}$. It follows from LaSalle's Principle (e.g.,  \cite[Theorem 8.3.1]{wiggins}) that $L(t)$ converges to the set $\{ L \in \mathcal{L}_0 \mid \nabla_{L} J(L) = 0 \}$. By Theorem \ref{thm:grad1_main}, this is 
\[ \{ L \in \mathcal{L}_0 \mid W_o(L) K(L) = 0 \}. \] 

But the observability of $(A,CA)$ implies the observability of $(A-LCA, CA)$ (otherwise, an unobservable mode of the latter is immediately an unobservable mode of the former), which means that $W_o(L)$ is strictly positive definite. We conclude that $L(t)$ converges to the set $K(L)=0$, which is a singleton with $L_{KF}$ as its only member.

To prove the convergence rate in the theorem statement, first observe that by compactness of $\mathcal{L}_0$, we have that $\kappa > 0$. To show that $c < \infty$, it suffices to analyze the behavior of the ratio in the definition of $c$ in an  \ao{open neighborhood ${\cal B} \subset \cL_0$ containing $L_{KF}$, since outside of of this neighborhood, the ratio is finite by compactness of $\mathcal{L}_0$} (and the fact that $L_{KF}$ is only point with $K(L_{KF})=0$)\footnote{\ao{More formally, suppose that we have that \[ \sup_{  L \in \mathcal{L}_0 \cap \cB \setminus \{L_{KF}\}}  \frac{J_{\rm innov}(L) - J_{\rm innov}(L_{KF})}{||K(L)||_F^2} < \infty.\] To further establish that $c<\infty$, it is enough to show finiteness of the same ratio on the complement $ \cB^c \cap \mathcal{L}_0$. But since the denominator is uniformly bounded away from zero on the compact set $ \cB^c \cap \cL_0$, the ratio  is indeed finite because it is the supremum of a continuous function on a compact set.}}.
Since $J_{\rm innov}(L)$ is twice continuously differentiable around $L_{KF}$ (indeed, $P(L)$ solves a linear equation, whose solution exist as long as $F(L)$ is stable, i.e. for $L \in \cS$), and $\nabla_L J_{\rm innov}(L_{KF}) = 0$, we have
\[ J_{\rm innov}(L) - J_{\rm innov}(L_{KF}) = \frac{1}{2} \langle \Delta L, \nabla^2 J_{\rm innov}(L_{KF})[\Delta L] \rangle + o(\|\Delta L\|_F^2), \] where $\Delta L := L-L_{KF}.$ 
Similarly, $K(L)$ is continuously differentiable, and $K(L_{KF}) = 0$. So,
\[ K(L) = d K(L_{KF})\cdot \Delta L + o(\|\Delta L\|_F), \]
where $d K(L_{KF})\cdot \Delta L =:D_K(\Delta L)$ denotes the differential  of $K$ at $L_{KF}$ applied to $\Delta L$. 
Then,
\[ \|K(L)\|_F^2 = \|D_K(\Delta L)\|_F^2 + o(\|\Delta L\|_F^2). \]
The limit of the ratio in the definition of $c$ as $L \to L_{KF}$ (i.e., $\Delta L \to 0$) is:
\begin{equation} \label{eq:climit} \lim_{\Delta L \to 0} \frac{\frac{1}{2} \langle \Delta L, \nabla^2 J_{\rm innov}(L_{KF})[\Delta L] \rangle_F}{\|D_K(\Delta L)\|_F^2}. 
\end{equation}
Since $\nabla_L J_{\rm innov}(L) = -2 W_o(L) K(L)$,  differentiating at $L_{KF}$, we obtain for the second derivative 
\begin{multline*} \nabla^2 J_{\rm innov}(L_{KF})[\Delta L]  = -2 W_o(L_{KF}) D_K(\Delta L) -2 (dW_o(L)\cdot \Delta L) K(L_{KF})  = -2 W_o(L_{KF}) D_K(\Delta L)
\end{multline*}
where we recall that  $K(L_{KF})=0$.
{We can therefore conclude that  $$\sup_{\|\Delta L\|_F=1} \frac{\frac{1}{2} {\langle \Delta L, \nabla^2 J_{\rm innov}(L_{KF})[\Delta L] \rangle_F}}{\|D_K(\Delta L)\|_F^2} < \infty,$$ by arguing as follows. First, observe that both the numerator and denominator are quadratic forms in $\Delta L$; next, note that if the denominator is zero, then $D_K(\Delta_L) =0$, which then implies the numerator is zero. Owing to Lemma \ref{lemma:ratio_quadratic}, there exists $\beta <\infty$ so that  
$$ \frac{1}{2} \langle \Delta L, \nabla^2 J_{\rm innov}(L_{KF})[\Delta L] \rangle_F \leq \beta \|D_K(\Delta L)\|_F^2  $$
which shows finiteness of the supremum and of $c$.  
}

We thus have
\[
  J_{\mathrm{innov}}(L)-J_{\mathrm{innov}}(L_{\mathrm{KF}})
   \;\le\;c\,\|K(L)\|_F^{2}. \]
Furthermore, from the definition of $\kappa$ and the fact that $\nabla J_{\mathrm{innov}} = -2W_o(L)D_K(L)$, we have
\[ 
  \|\,\nabla J_{\mathrm{innov}}(L)\|_F^{2} \;\ge\;4\kappa^{2}\,\|K(L)\|_F^{2}.
\] 
We can put these together to obtain 
\[ \| \nabla J_{\mathrm{innov}}(L)\|_F^2 \geq 4 \frac{\kappa^2}{c} (J(L(t)) - J(L_{KF}).\] Thus 
\[ \frac{d}{dt} (J(L(t)) - J(L_{KF})) \leq -4 \frac{\kappa^2}{c} (J(L(t)) - J(L_{KF})),\] 
leading to the geometric estimate in the theorem statement. 
\end{proof}

\section{The Coercivity Lemma} \label{sec:coercivity_issue}

The coercivity property is essential for analysis of descent in the previous section. In this section, we provide a complete proof. 

Proving this turns out to be surprisingly challenging. The problem is that 
\begin{itemize} 
\item Coercivity is not true when $(A,CA)$ is not observable; see Example \ref{ex:spurious} where the level sets are not compact \item Even when $(A,CA)$ is observable, the introduction of the gain $L$ can cause the system to asymptotically lose observability (in that the observability Gramian can have eigenvalues which go to zero as $\|L\|_F \rightarrow \infty$). 
\end{itemize} 
Intuitively, the system can have modes that asymptotically become unobservable as the gain $L$ gets large, and one must show that the noise cannot ``hide within'' these modes.

We need a careful argument ruling this out. Our first step is to  relate the coercivity of $J_{\rm innov}(L)$ to a simpler quantity which we will call $J_{\rm pred}$, which corresponds to the innovations loss in an alternative form of the Kalman filter which makes one-step ahead predictions. Formally, let $$\tilde{F}(L) = A - LC,$$ and denote the steady-state estimation error covariance by $\tilde P(L)$. It is the unique positive definitie solution of:
\[ \tilde{P}(L) = \tilde{F}(L) \tilde{P}(L) \tilde{F}(L)^\T + Q_w + L R_v L^\T, \]
assuming $\rho(\tilde{F}(L)) < 1$. We introduce the notation $\tilde \cS = \{ L \mid \rho(A-LC) < 1 \}$.

We then introduce the loss function 
\begin{equation} \label{eq:jpred_def}
J_{\rm pred}(L) = \Tr(C\tilde{P}(L)C^\T + R_v),
\end{equation}
which is the innovations loss associated with the linear filter of the form, 
\[ \tilde{x}_{t+1}= A \tilde{x}_t + L (y_t - C \tilde{x}_t). \] We skip the details of this derivation here and refer the reader to~\cite{talebi2023data}.

We next restate the coercivity lemma, giving it now an additional part, and provide a proof. 

\begin{lemma}[Expanded Coercivity Lemma]\label{thm:coercivity_J_G}
\begin{enumerate} \item Suppose $(A,C)$ is observable, $Q_w \succ 0$, and $R_v \succ 0$. 
If a sequence $L_j \in \tilde \cS$ satisfies either $\rho(A-L_jC) \to 1^-$ or $\|L_j\|_F \to \infty$, then $J(L_j) \to \infty$.
\item Suppose Assumption   \ref{assum:obs_main} holds.
For any sequence $L_j \in \mathcal{S}$ such that either $\|L_j\|_F \to \infty$ or $L_j$ approaches the boundary of $\mathcal{S}$ (i.e., $\rho(F(L_j)) \to 1^-$), we have $J_{\rm innov}(L_j) \to \infty$.
\end{enumerate}
\end{lemma}

\begin{proof}
Let us introduce the shorthand $M_L = A-LC$. Then steady-state error covariance is
\[ \tilde{P}(G) = \sum_{j=0}^\infty M_L^j S_L (M_L^\T)^j, \quad \text{where } S_L = Q_w + LR_vL^\T. \]
By our assumptions, $S_L \succ 0$. The cost then satisfies
\begin{equation} \label{eq:jlowerbound} J_{\rm pred}(L) = \sum_{k=0}^\infty \Tr\left((C M_L^k) S_L (C M_L^k)^\T\right) \geq \lambda_{\rm min}(S_L) \sum_{k=0}^\infty \|C M_L^k\|_{\mathrm{F}}^2 = \lambda_{\min}(S_L) {\rm Tr}(X_o(L)), 
\end{equation} where $X_o(L)$ is the observability Gramian for $(M_L,C)$.

\textbf{Case 1: $L_j$ remains bounded and $\rho(M_{L_j}) \to 1^-$.}
In this case, $\lambda_{\min}(S_{L_j})$ is bounded below by $\lambda_{\min}(Q_w) > 0$. Observe that 
$X_o(L)$ solves $X_o(L) = M_L^\T X_o(L) M_L + C^\T C$.
Let $\lambda_j$ be an eigenvalue of $M_{L_j}$ with $|\lambda_j| = \rho(M_{L_j}) \to 1^-$, and $w_j$ be the corresponding normalized eigenvector. Then from the previous equation we get 
$$w_j^* X_o(L_j) w_j (1 - |\lambda_j|^2) = \|C w_j\|^2,$$
implying 
\begin{equation} \label{eqxtracelowerbound} \Tr(X_o(L_j)) \ge w_j^* X_o(L_j) w_j \geq \frac{\|C w_j\|^2}{1 - |\lambda_j|^2}.
\end{equation}

Since $(A,C)$ is observable, it follows that $(M_{L_j},C) = (A-L_jC, C)$ is observable for any $L_j$ (indeed, any unobservable mode of $(M_{L_j},C)$ is an unobservable mode of $(A,C)$). Thus $C w_j \neq 0$. By the assumption that $L_j$ stay in a bounded set, it also follows that $\inf_j \|Cw_j\|^2 > 0$ (otherwise an accumulation point of $L_j$ would be a $G'$ such that $(M_{G'},C)$ is not observable). This observation along with~\eqref{eqxtracelowerbound})and~\eqref{eq:jlowerbound} completes the proof in this case.


\smallskip 
\textbf{Case 2: $\|L_j\|_F \to \infty$.}
In this case, let  $\alpha_j = \|L_j\|_F \to \infty$ and $\bar{L}_j = L_j / \alpha_j$, so $\|\bar{L}_j\|_F=1$. By passing to a subsequence, we can assume without loss of generality that $\bar{L}_j \to L^*$ with $\|L^*\|_F=1$.
From~\eqref{eq:jlowerbound}, we have 
\[ J_{\rm pred}(L_j) \ge \sum_{k=0}^\infty \Tr\left(C M_{L_j}^k (\alpha_j^2 \bar{L}_j R_v \bar{L}_j^\T) (M_{L_j}^\T)^k C^\T\right) = \alpha_j^2 \sum_{k=0}^\infty \Tr\left( (C M_{L_j}^k \bar{L}_j) R_v (C M_{L_j}^k \bar{L}_j)^\T \right). \]

Let \begin{equation}\label{eq:deflj}L_k^{(j)} = C M_{L_j}^k \bar{L}_j = C (A-\alpha_j \bar{L}_j C)^k \bar{L}_j.
\end{equation}
With this notation,  $$J_{\rm pred}(L_j) \ge \alpha_j^2 \lambda_{\min}(R_v) \sum_{k=0}^\infty \|L_k^{(j)}\|_{\mathrm{F}}^2.$$
Now let us suppose that  $J(L_j)$ is bounded above; from this we will obtain a contradiction. 

Under this assumption, using the previous equation, we obtain
\begin{equation} \label{eq:L_k_bound}
\|L_k^{(j)}\|_{\mathrm{F}} = O(1/\alpha_j) \mbox{ for all } k,j \geq 0.
\end{equation}
Now the definition of $L_k^{(j)}$ can be expanded in the following way. Starting from the general matrix relation 
\[ M^k - F^k =  \sum_{s=0}^{k-1} M^s (M-F) F^{k-1-s},\]
we use $M=A - \alpha_j \bar{L}_jC$ and $F = A$ to obtain 
\begin{equation}
\label{eq:specific_matrix_power_expansion}
(A - \alpha_j \bar{L}_j C)^k = A^k - \sum_{s=0}^{k-1} (A - \alpha_j \bar{L}_j C)^s (\alpha_j \bar{L}_j C) A^{k-1-s}.
\end{equation}
Multiplying on the left by $C$ and on the right $\bar{L}_j$ and using~\eqref{eq:deflj} we obtain:
\begin{equation} \label{eq:lrec}
L_k^{(j)} = C A^k \bar{L}_j - \alpha_j \sum_{s=0}^{k-1} C (A-\alpha_j \bar{L}_j C)^s \bar{L}_j C A^{k-1-s} \bar{L}_j = C A^k \bar{L}_j - \alpha_j \sum_{s=0}^{k-1} L_s^{(j)} C A^{k-1-s} \bar{L}_j. 
\end{equation}
Since $(A,C)$ is observable and $L^* \neq 0$, there must be a smallest integer $p \ge 0$ such that $C A^p L^* \neq 0$. We will show that under the assumption that $J$ is bounded, no such $p\geq 0$ exists, thus yielding a contradiction.

Assume that $p=0$, then since $L_0^{(j)} = C \bar{L}_j$, we have that  as $j \to \infty$, $L_0^{(j)} \to C L^* \neq 0$. This contradicts~\eqref{eq:L_k_bound}.

In order to deal with the case $p > 0$,  
we need the following auxiliary claim: 
Let $$\Delta_j \coloneqq \bar{L}_j - L^*.$$ We clearly have that  $\lim_{j \to \infty}\Delta_j = 0$.
We claim that  $$\|C A^k \Delta_j\|_F = O \left( 1/\alpha_j \right),$$ for $0 \le k < p$.
We prove the claim by induction.

\begin{description}
    \item[Base case ($k=0$):]   $L_0^{(j)} = C \bar{L}_j = C(L^*+\Delta_j) = C \Delta_j$. 
Since we have shown above that $\|L_0^{(j)}\|_F = O(1/\alpha_j)$, this implies $\|C \Delta_j\|_F = O(1/\alpha_j)$. 
\item[Inductive step ($0<k<p$):]
Assume for some $k < p-1$ that $\|C A^s \Delta_j\|_F = O(1/\alpha_j)$ for all $0 \le s \le k$.  Using the fact that $CA^sL^*=0$ for all $0 \leq s <p$, we have from~\eqref{eq:lrec} that
\begin{multline}\label{eq:inter1} C A^k \bar{L}_j = CA^k(L^*+\Delta_j)=CA^k\Delta_j = L_k^{(j)} + \alpha_{j} \sum_{s=0}^{k-1} 
L_s^{(j)} (C A^{k-1-s} (L^*+\Delta_j))\\= L_k^{(j)} + \alpha_{j} \sum_{s=0}^{k-1} 
L_s^{(j)} (C A^{k-1-s} \Delta_j).
\end{multline} 
Since $L_{k}^{(j)}$ and $L_s^{(j)}$ are $O(\alpha_j^{-1})$ by~\eqref{eq:L_k_bound}, and since all the terms $C A^{k-1-s} \Delta_j$ are $O(\alpha_j^{-1})$ by the inductive hypothesis, we conclude that $\|C A^k \Delta_j\|_F = O(1/\alpha_j)$, concluding the proof of the claim.
\end{description}
Returning to the main thread of the proof, assume that $p>0$.  Since $CA^{p-1-s}L^*=0$ for $s\geq 0$, we have  \begin{equation} \label{eq:p} L_p^{(j)} = C A^p L^*+CA^p \Delta_j - \alpha_j \sum_{s=0}^{k-1} L_s^{(j)} (C A^{k-1-s} \Delta_j)=C A^p L^* + o(1) + O(\alpha_j^{-1})
\end{equation}
where $CA^p \Delta_j$ is the $o(1)$ term and the sum in~\eqref{eq:p} is the $O(\alpha_j^{-1})$ by the above claim. Since $C A^p L^* \neq 0$, this contradicts~\eqref{eq:L_k_bound} and  completes the proof of part 1. 

To prove part 2, observe that a priori estimate $\hat{x}_{t+1}^-(L)$ from our main filter follows the recursion:
\[ \hat{x}_{t+1}^-(L) = A\hat{x}_t(L) = A(\hat{x}_t^-(L) + L\,\delta_t(L)) = A(\hat{x}_t^-(L) + L(y_t - C\hat{x}_t^-(L))). \]
As a consequence,  $P^-(L) = \tilde{P}(AL)$.
Since from~\eqref{eq:pminus}) and~\eqref{eq:jinnov_def}, we have  $$J_{\rm innov}(L) = \Tr(C P^-(L) C^\T + R_v),$$ we conclude that $J_{\rm innov}(L) = J_{\rm pred}(AL)$. Since observability of $(A,CA)$ implies invertibility of $A$ and also observability of $(A,C)$, we have that the coercivity result of part 1 implies that $J_{\rm innov}(L)$ is coercive under the assumptions of part 2. 
\end{proof}

\section{Conclusion}

This paper has four main contributions. First, we derived an  interpretable formula for the gradient of the Kalman gain. Second, we showed that spurious stationary points exist under the assumption that the system is observable in the standard Kalman filter formulation (as opposed to one-step look-ahead formulation considered in the earlier literature), preventing gradient descent from recovering the Kalman gain. Third, we identified a non-standard observability condition which is sufficient for gradient descent to recover the Kalman gain. Finally, we gave an interpretable convergence rate for gradient descent under the same condition. 

One limitation of our results is that  Theorem \ref{thm:convergence_main} contains the constants $\kappa$ and $c$ which are in turn defined through optimization over a level set of  the function $J_{\rm innov}(L)$. It is an open question whether one can obtain a convergence bound that depends in a clean way only on $A,C,L_0$ without optimization over a level set. This was accomplished in \cite{li2023policy} for a slightly different objective function, and it is not clear if it can be achieved for the innovations loss we consider here.

\bibliographystyle{plain} 
\bibliography{bibliography} 

\end{document}